\documentclass[11pt]{article} 
\usepackage{latexsym}
\usepackage{amsfonts, manfnt}
\usepackage{amsthm}
\usepackage{amsmath}
\usepackage{amssymb}
\usepackage{setspace}
\usepackage{eucal}
\usepackage{graphicx}
\usepackage[all, poly, knot]{xy}
\usepackage{verse, play}
\input xy
\xyoption{all}
\xyoption{knot}
\xyoption{arc}

\newtheorem{theorem}{Theorem}

\theoremstyle{definition}

\theoremstyle{remark}

\numberwithin{equation}{section}
\theoremstyle{plain}

\newtheorem{case}{Case}

\begin{document}
\thispagestyle{empty}
 \title{A Midsummer Knot's Dream}
 \author{A. Henrich, N. MacNaughton, S. Narayan, O. Pechenik, R. Silversmith, J. Townsend}

\maketitle 

\section*{Prologue}
\poemlines{5}

\begin{verse}
\verselinenumbersleft
\verselinenumfont{\itshape \tiny} 
Five undergrads and one new PhD,\\
In Williamstown, where we lay our scene,\\
From ancient math, break to a new theory,\\
Where playing games makes knots from shadows glean.\\
From forth the fertile brains of these fine friends,\\
A pair of knotty games are brought to life;\\
In the few weeks before the summer ends,\\
With joy and happiness the REU is rife,\\
As knotty theorems six companions prove,\\
And battles of wits two feuding players wage.\\
One tangle would, one a knot would remove---\\
This math is now the content of our page.\\
The which, if you with patient mind attend,\\
To what is writ, new theorems please append.
\end{verse}

\section*{Act I: What is past is prologue}

A \emph{knot} is an embedding of a circle in $\mathbb{R}^3$. We say two knots are equivalent if there is a way to deform one into the other without breaking the circle (or passing it through itself). You can think of a knot as a tangled piece of rope with its ends glued together. Anything that can be done to the rope short of cutting it (e.g. stretching, pulling, tangling) will preserve the knot type. Any knot equivalent to the unit circle in a plane of $\mathbb{R}^3$ is called the \emph{unknot}.

We typically study knots by looking at their diagrams, such as those in Figure~\ref{Diagram}. Early in the $20^\text{th}$ century, Kurt Reidemeister showed that two knot diagrams represent the same knot if and only if one diagram can be obtained from the other by a sequence of the moves pictured in Figure~\ref{ReidMoves}, now known as \emph{Reidemeister moves}.

\begin{figure}[h]
\begin{center}\includegraphics[height=1in]{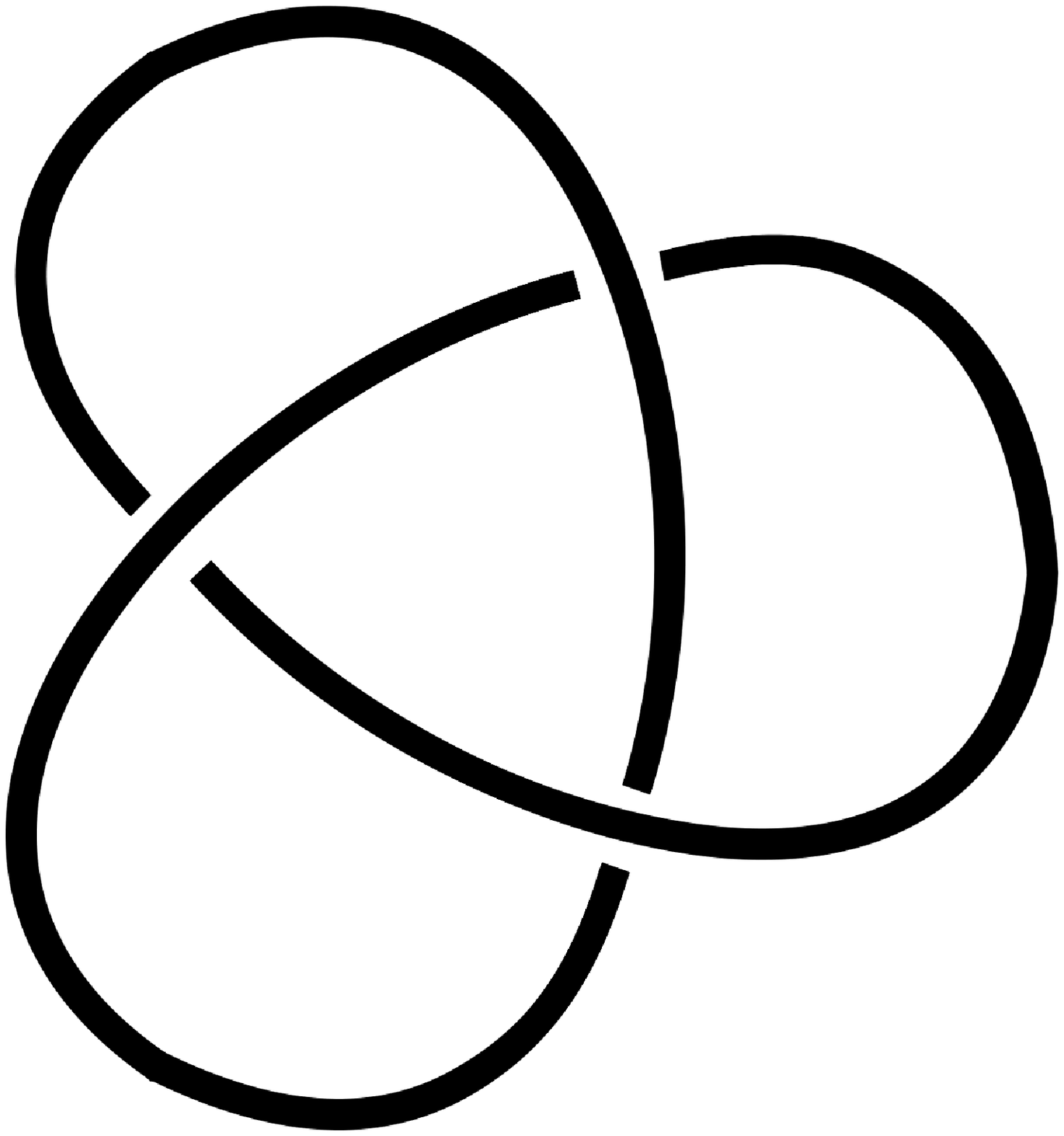}\hspace{.2in}
\includegraphics[height=1in]{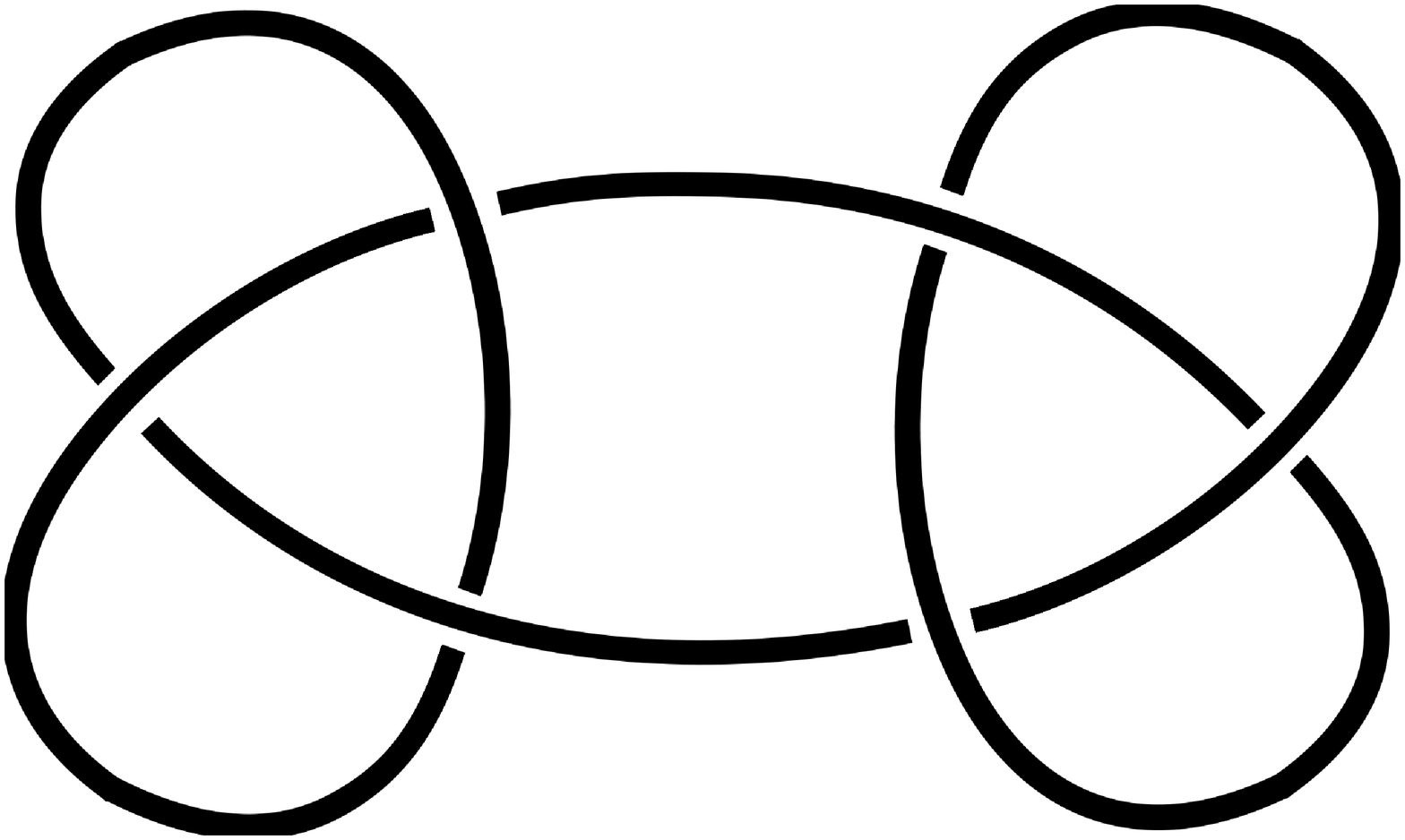}\hspace{.2in}
\includegraphics[height=1.2in]{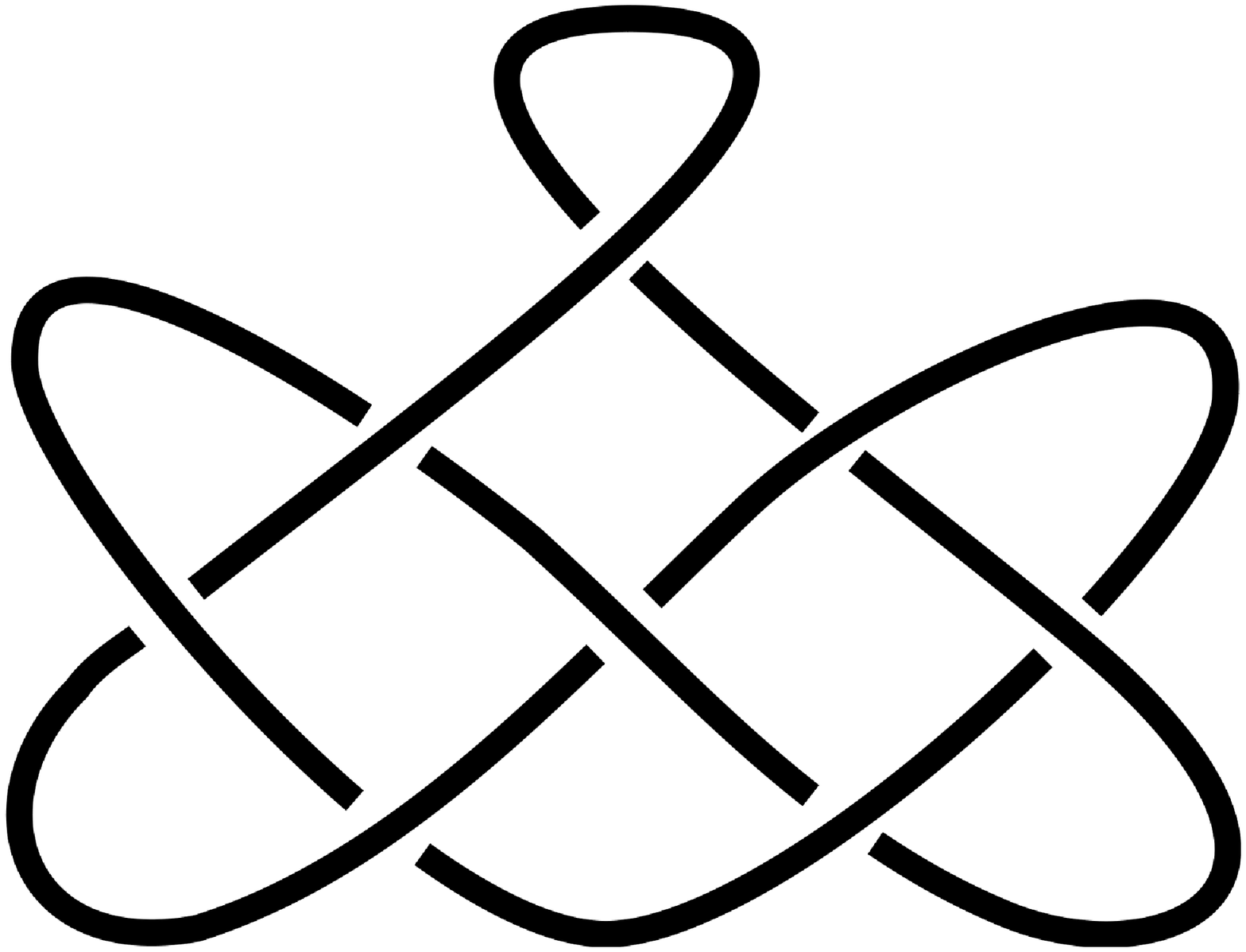}
\end{center}
\caption{An assortment of knot diagrams}\label{Diagram}
\end{figure}

\begin{figure}[h]
\[
\xy
(-50,-10)*{}; (-50,10)*{} **\dir{-};
{\ar@{<->}(-47,0)*{}; (-42,0)*{}}; ?(.75)*dir{>}+(-1, 3)*{(0)};
(-35,-10)*{}="A";
(-35,10)*{}="B";
"A"; "B" **\crv{(-30,-5) & (-40,-2) & (-32,1) & (-38,6)};
(-10,-10)*{}; (-10,10)*{} **\dir{-};
{\ar@{<->}(-7,0)*{}; (-2,0)*{}}; ?(.75)*dir{>}+(-1, 3)*{(1)};
(5,10)*{}="C";
(5, -10)*{}="D";
(5,5)*{}="C'";
(5,-5)*{}="D'";
"C"; "C'" **\dir{-};
"D"; "D'" **\dir{-};
(8,0)*{}="MB";
(12,0)*{}="LB";
"C'"; "LB" **\crv{(6, -4) & (12,-4)}; \POS?(.25)*{\hole}="2z";
"LB"; "2z" **\crv{(13, 6) & (7, 6)};
"2z"; "D'" **\crv{(5,-3)};
(30,10)*{}="E";
(30, -10)*{}="F";
"E"; "F" **\crv{(34, 0)};
(40,10)*{}="G";
(40, -10)*{}="H";
"G"; "H" **\crv{(36, 0)};
{\ar@{<->}(43,0)*{}; (48,0)*{}}; ?(.75)*dir{>}+(-1, 3)*{(2)};
(50,10)*{}="I";
(50, -10)*{}="J";
"I"; "J" **\crv{(62,0)}; \POS?(.25)*{\hole}="2x"; \POS?(.75)*{\hole}="2y";
(60,10)*{}="K";
(60, -10)*{}="L";
"K"; "2x" **\crv{(55,7)};
"2x"; "2y" **\crv{(50, 0)};
"2y"; "L" **\crv{(55, -7)};
\endxy
\]
\[
\xy
(75,10)*{}="AT";
(85, -10)*{}="AB";
(85, 10)*{}="BT";
(75, -10)*{}="BB";
(71, -2)*{}="CL";
(89, -2)*{}="CR";
"CL"; "CR" **\crv{(80, -10)}; \POS?(.35)*{\hole}="a"; \POS?(.65)*{\hole}="b"; 
"AT"; "b" **\crv{}; \POS?(.35)*{\hole}="c";
"b"; "AB" **\crv{};
"BB"; "a" **\crv{};
"a"; "c" **\crv{};
"c"; "BT" **\crv{};
{\ar@{<->}(92,0)*{}; (98,0)*{}}; ?(.75)*dir{>}+(-1, 3)*{(3)};
(105,10)*{}="A'T";
(115, -10)*{}="A'B";
(115, 10)*{}="B'T";
(105, -10)*{}="B'B";
(101, 2)*{}="C'L";
(119, 2)*{}="C'R";
"C'L"; "C'R" **\crv{(110, 10)}; \POS?(.35)*{\hole}="a'"; \POS?(.65)*{\hole}="b'"; 
"A'T"; "a'" **\crv{};
"a'"; "A'B" **\crv{};\POS?(.65)*{\hole}="c'";
"B'B"; "c'" **\crv{};
"c'"; "b'" **\crv{};
"b'"; "B'T" **\crv{};
\endxy
\]
\caption{Reidemeister moves}\label{ReidMoves}
\end{figure}

Take another look at Figure~\ref{Diagram}. What would happen if we lost information at some crossings about which strand is on top? Would we still be able to tell if the knot is knotted (i.e., non-trivial), or the unknot (trivial)? In some cases, we can lose a fair amount of information about our knot without sacrificing our ability to tell if it is knotted. 

We will refer to diagrams of knots containing only partial crossing information as \emph{pseudodiagrams}, a concept developed by Hanaki \cite{hanaki} and ourselves \cite{us}. A crossing that is undetermined is called a \emph{precrossing}. If a pseudodiagram contains no under-/over-strand information about any of its crossings (that is, if all crossings are precrossings), then we will refer to it as a \emph{shadow}. Figure~\ref{Pseudo} illustrates this terminology.

\begin{figure}[h]
\begin{center}\includegraphics[height=1in]{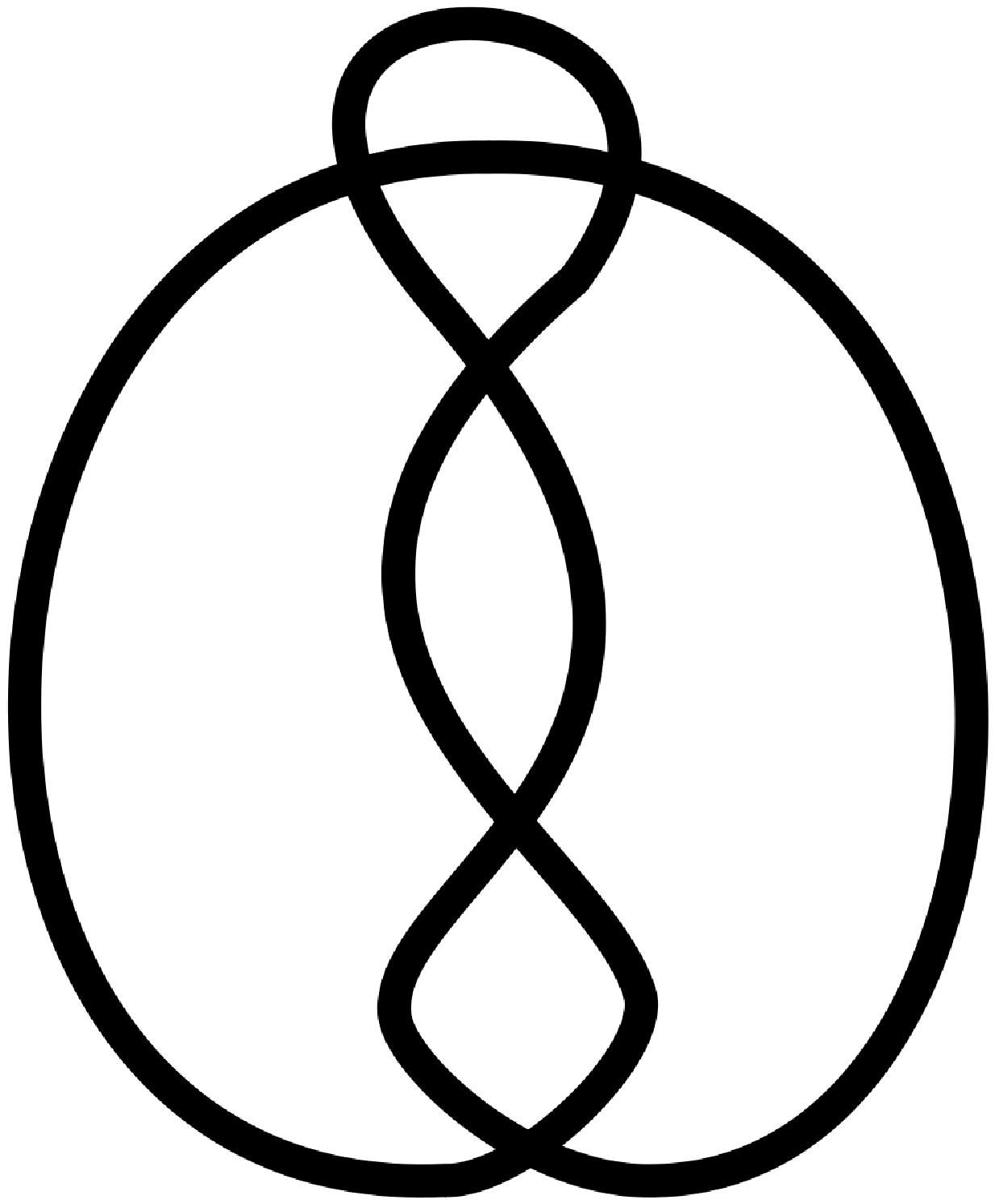}
\includegraphics[height=1in]{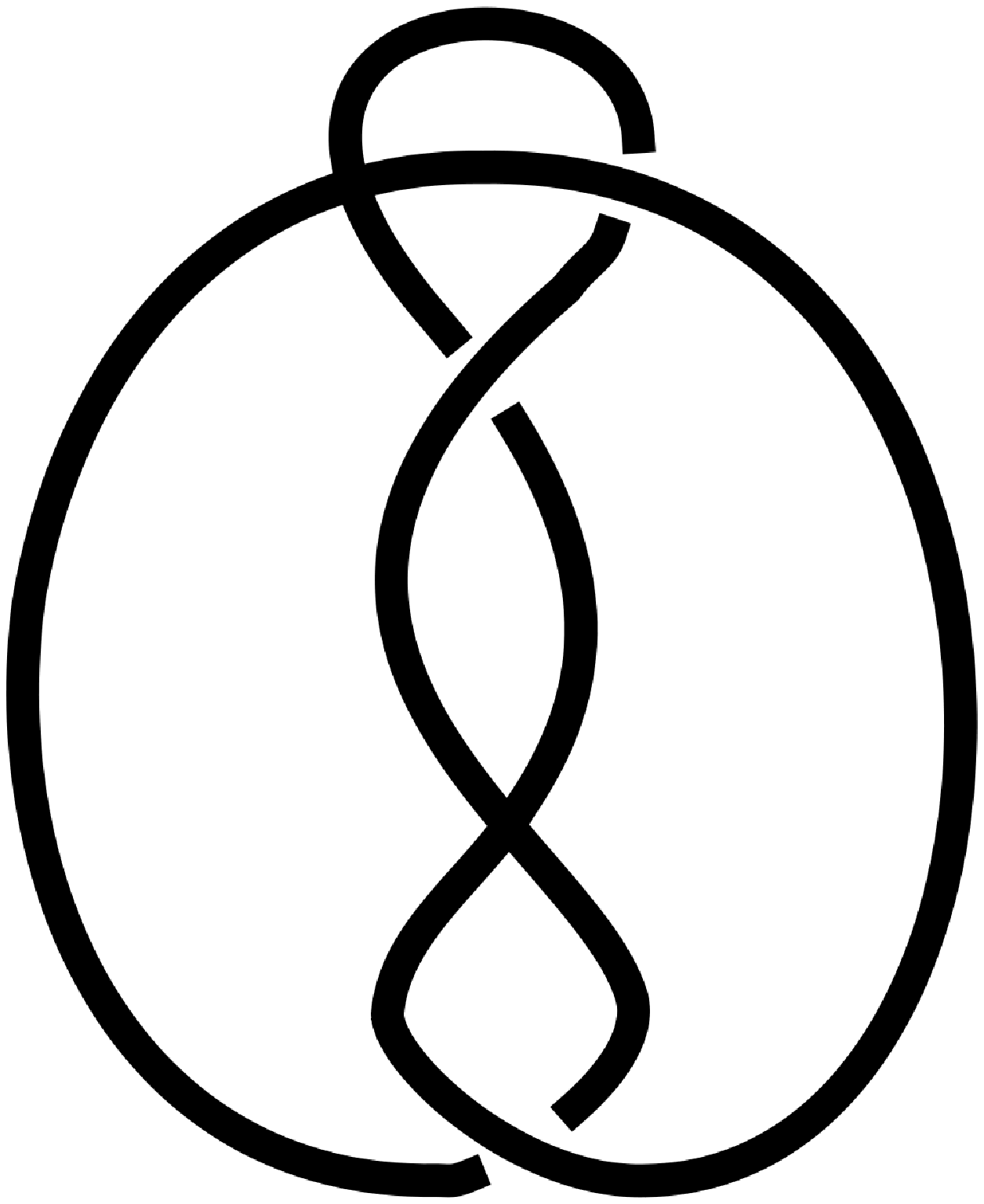}
\includegraphics[height=1in]{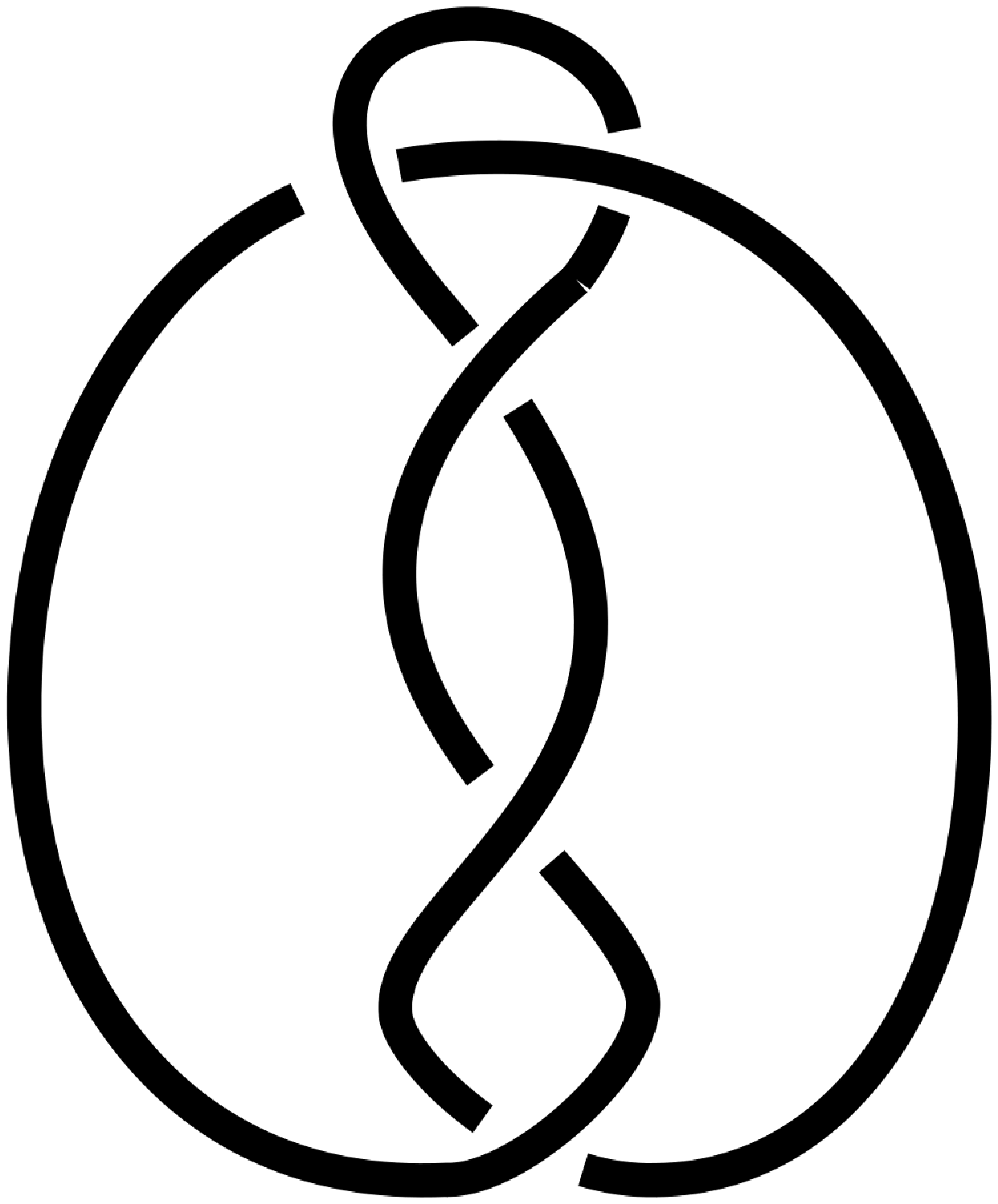}
\end{center}
\caption{Three pseudodiagrams: the first is a shadow, and the last is a knot diagram. The last is knotted, whereas the first two may or may not be trivial, depending on how the precrossings are resolved.}\label{Pseudo}
\end{figure}

Pseudodiagrams are good for lots of things, such as modeling DNA. But they're also really good for playing games! We will describe two pseudodiagram games that we invented, and prove some basic results about them. There is a lot of room to play around with knot games, and we hope knot games will lead to interesting results in traditional not-game knot theory.

\section*{Act II: To Knot or Not to Knot}

In our first game, \emph{To Knot or Not to Knot}, two players, Ursula and King Lear, are given a shadow and will take turns resolving crossings. {\bf U}rsula attempts to turn the shadow into an {\bf u}nknotted pseudodiagram, while {\bf K}ing Lear attempts to create a {\bf k}notted pseudodiagram. Let's sit in on a game to get an idea of how this works. Our players will play on the Twelfth Knot (or rather, the shadow of the standard diagram of the twelfth knot in the knot table, \cite{KnotTable}). The pre-game setup is in Figure \ref{examplegame}A.

The players have agreed that Ursula go first. She makes the move in Figure \ref{examplegame}B. At this point, it is clear that the pseudodiagram is not yet determined to be knotted or unknotted. King Lear then makes the move in Figure \ref{examplegame}C. Again, the pseudodiagram is not determined. Ursula, wanting to simplify the knot so that it is more likely to be trivial, makes the move in Figure \ref{examplegame}D. She has created a situation in which a Reidemeister 2 move can simplify the diagram by removing two crossings. Despite the fact that she has guaranteed the knot is one that can be drawn with five or fewer crossings, she has still not resolved the pseudodiagram to be the unknot. Hence, King Lear still has a chance to win. 

\begin{figure}[h]
\begin{center}\includegraphics[height=3in]{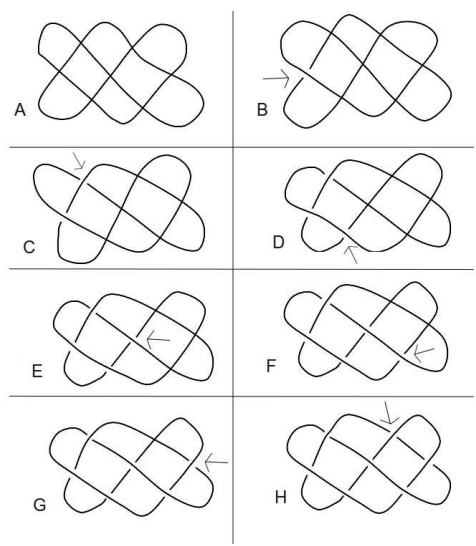}
\end{center}
\caption{Possible play of \emph{To Knot or Not to Knot}. A is the setup; B through G are successive stages of the game; Ursula wins by moving from G to H.}\label{examplegame}
\end{figure}

Suppose King Lear decides on the move in Figure \ref{examplegame}E. You may check that there is still a way to determine the remaining precrossings so that the diagram is knotted, as well as a way to resolve them to form the unknot. In an attempt to simplify the knot, Ursula tries the move in Figure \ref{examplegame}F.

Resolving the remaining crossings will create either the unknot or the figure-eight knot. You can see the figure-eight knot as knot 4.1 in \cite{KnotTable}. King Lear attempts to resolve another precrossing so that it is possible to realize the figure-eight knot in the diagram. Clever Ursula sees that her final decision is critical! She has the power to knot or to unknot the pseudodiagram. Naturally, she decides to create the unknot. Do you see how to use Reidemeister moves to unknot the diagram in Figure \ref{examplegame}G?

Ursula has won! Clever girl!

Perhaps you would like to see what might have happened had King Lear been the first player to move. Could he have guaranteed a win for himself? What might have happened if they had played on a different shadow? At this point, the authors encourage you take an intermission, find a colleague, friend, student, or dean---and play a game of your own!

\section*{Act III: A bliss in proof}
Given a shadow and a player to go first, either King Lear or Ursula has a winning strategy. We investigate, for certain classes of shadows, which player this is.

We now play the game on an infinite family of shadows. This may take a while.  The shadows we consider are from a family of knots known as \emph{twist knots}.  A shadow of one of these knots is illustrated in Figure~\ref{twistknot}.

\begin{figure}[h]
\begin{center}
\includegraphics[width=.3\textwidth]{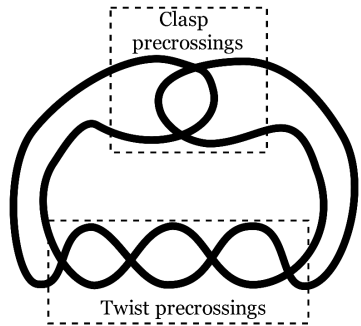}
\caption{The shadow of a twist knot.}
\label{twistknot}
\end{center}
\end{figure}

\begin{theorem}\label{twisties}
Suppose \emph{To Knot or Not to Knot} is played on a standard shadow of a twist knot with $n$ twist precrossings. Then, King Lear has a winning strategy if and only if $n$ is even and King Lear plays second.
\end{theorem}
\begin{proof}
Observe that if King Lear is first of the two players to play in the clasp region, then Ursula wins by playing on the other clasp precrossing. This is because if the clasp can be taken apart by a Reidemeister 2 move, then the entire pseudodiagram is necessarily unknotted, as the rest of the precrossings can be untwisted by a sequence of Reidemeister 1 moves.

\begin{case}
King Lear plays first and $n$ is even.
\end{case}
Ursula has the following winning strategy. She will always play at a twist precrossing, unless King Lear has just played in the clasp. Since there are an even number of twist precrossings, King Lear necessarily plays first in the clasp. Then, Ursula plays on the other clasp precrossing and wins.

\begin{case}
King Lear plays second and $n$ is odd.
\end{case}
Ursula's strategy from Case 1 ensures victory here as well.

\begin{case}
King Lear plays first and $n$ is odd.
\end{case}
Ursula should only play in the clasp if King Lear has just played there or if there are no remaining twist precrossings. 

Whenever King Lear plays on a twist precrossing, Ursula plays on an adjacent twist precrossing, so that the two resolved crossings can be removed by a Reidemeister 2 move. This reduces the game to playing on a twist shadow with $n-2$ twist precrossings. Therefore, it suffices to show that Ursula has a winning strategy when $n =1$. 

We have seen that Ursula can win if King Lear plays first in the clasp. Suppose instead that King Lear's first move is to resolve the unique twist precrossing. Then, Ursula can resolve the lower clasp precrossing so that the two resolved crossings can be removed by a Reidemeister 2 move. Now, there is only one crossing left, and so the diagram is necessarily unknotted.

\begin{case}
King Lear plays second and $n$ is even.
\end{case}
King Lear has a winning strategy. If at any time Ursula plays in the clasp, King Lear should also play in the clasp so that those two crossings \emph{cannot} be removed by a Reidemeister 2 move.

Whenever Ursula plays on a twist precrossing, King Lear can play on an adjacent twist precrossing, so that the two resolved crossings can be removed by a Reidemeister 2 move. This reduces the number of twist precrossings by two. Therefore, it is sufficient to show that King Lear has a winning strategy when $n=2$. 

When Ursula plays on one of the two twist precrossings, King Lear should play on the other, so that these two crossings \emph{cannot} be removed by a Reidemeister 2 move. This produces a diagram of either the trefoil (the leftmost knot in Figure \ref{Diagram}) or the figure-eight knot. Congratulations King Lear, you finally beat her!
\end{proof}

Twist knots are an example of a large family of knots called \emph{rational knots.}  
Rational knots can be named by means of Conway notation. Conway notation records the number of consecutive horizontal twists, followed  by the number of consecutive vertical twists, followed by the number of twists in the next set of horizontal twists, etc. An example of Conway notation is illustrated in Figure \ref{conway}. A more rigorous discussion of rational knots can be found in \cite{knotbook}. Rational knots are more tangly than twist knots, but let's see if we can generalize our game strategies to this larger class of knots.

\begin{figure}[h]
\centering
\includegraphics[scale=.3]{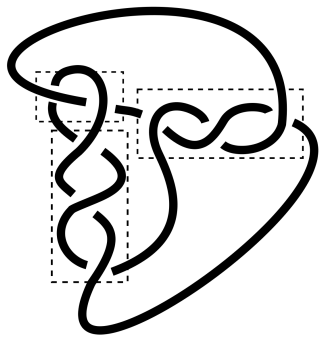}
\caption{A knot with Conway notation -2 3 3.}
\label{conway}
\end{figure}

\begin{theorem} For any shadow with Conway notation containing only even integers, the second player has a winning strategy, regardless of whether they wish to knot or unknot.
\end{theorem}
\begin{proof}
It suffices to show that the second player can reduce the game to one on a twist knot shadow with an even number of precrossings.

Until the shadow becomes a twist knot (there exist only two sets of twists, one of which has exactly two twists), the second player should apply the following strategy. Whenever the first player resolves a precrossing in some set of consecutive twists, the second player resolves an adjacent precrossing in that set so that these two resolved crossings can be eliminated by a Reidemeister 2 move.  Observe that the second player can always resolve such an adjacent precrossing as the number of twists in every set is even. 

By applying this strategy repeatedly and reducing the shadow by a Reidemeister 2 move at each stage, the second player ensures that a twist knot shadow with an even number of precrossings is eventually obtained. An even number of precrossings have been resolved, so the second player plays second on this twist knot shadow as well.  Since it was shown in Theorem \ref{twisties} that the second player has a winning strategy for such a game, it follows that the second player has a winning strategy on $S$.
\end{proof}

King Lear now worries that he always loses when he is first to move on a shadow. Nay, good my lord, be not afraid of shadows!
\begin{theorem} If a certain player has a winning strategy on some shadow $S$ when playing second, then that same player has a winning strategy playing first on the shadow $S'$, where $S'$ is derived from $S$ as shown in Figure \ref{nugatory}.
\end{theorem}
\begin{figure}[h]
\begin{center}
\includegraphics[width=.6\textwidth]{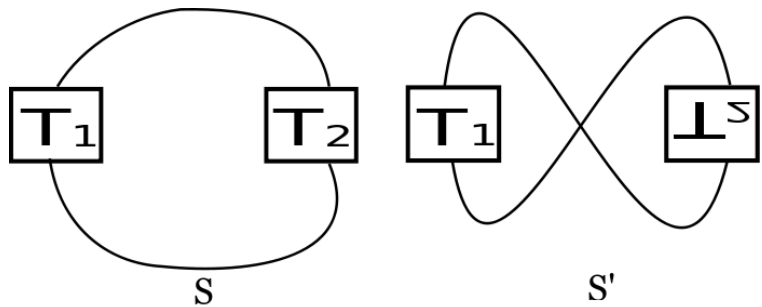}
\caption{The shadow $S'$ is obtained from $S$ by flipping $T_2$, introducing a new precrossing.}
\label{nugatory}
\end{center}
\end{figure}
\begin{proof}
The first player on the shadow $S'$ may play on the new precrossing $c$. This allows the shadow to be untwisted at $c$ to recreate the shadow $S$. The first player will now be the second player to play on $S$, and therefore has a winning strategy.
\end{proof}

We've seen games where King Lear can win playing first, and games where King Lear can win playing second.  We've also seen games where Ursula has a winning strategy regardless of who plays first.  King Lear wants to know: Is there any shadow on which he always has a winning strategy regardless of who plays first?  We'd like to know too.

\section*{Act IV: Much Ado About Knotting}

\emph{To Knot or Not to Knot} is a fascinating game. Nonetheless, we switch gears to describe a second game. \emph{Much Ado About Knotting} is also a game for two players, Portia and Nerissa, who play on a shadow. The twist is that in this game, the players have distinct allowable moves while sharing a common goal. On each turn, {\bf P}ortia is allowed to resolve a precrossing as a {\bf p}ositive crossing, while {\bf N}erissa must resolve precrossings as {\bf n}egative. What does this mean? To see how signs determine over-/under-strand information, we choose an orientation (i.e. a direction to travel around) for a given pseudodiagram. Given a particular crossing in the diagram, look at the picture so that both crossing strands are oriented upwards. If the strand pointing up and to the right is the over-strand, then we say that the crossing is positive. Otherwise, we say the crossing is negative. Figure~\ref{Signs} illustrates the two types of crossings. 
\begin{figure}[h]
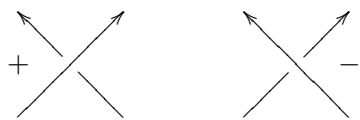

\[
\xy
(-54,-7)*{}; (-40, 7)*{} **\crv{}; \POS?(.5)*{\hole}="y"; ?(1)*\dir{>}+(-14, -7)*{+};
(-40,-7)*{}; "y" **\crv{};
"y"; (-54, 7)*{} **\crv{}; ?(0)*\dir{>}+(-2, -2)*{};
(-10,-7)*{}; (-24, 7)*{} **\crv{}; \POS?(.5)*{\hole}="x"; ?(1)*\dir{>}+(14, -7)*{-};
(-24,-7)*{}; "x" **\crv{};
"x"; (-10, 7)*{} **\crv{}; ?(0)*\dir{>}+(-2, -2)*{};
\endxy
\]
\caption{Sign of a crossing}\label{Signs}
\end{figure}
Now we know what the players are allowed to do, but what is their goal? Each player would like to move without resolving the diagram as knotted or unknotted. In effect, the player who decides whether or not the pseudodiagram is knotted loses. 

\begin{theorem}
Suppose \emph{Much Ado About Knotting} is played on a standard shadow of a twist knot with $n$ twist precrossings. Then, the first player has a winning strategy.
\end{theorem}
\begin{proof}
Since both players have the same objective and their moves are symmetric, we can assume, without loss of generality, that Portia plays first. If instead Nerissa plays first, she can win by using the strategy described for Portia with the `positives' and `negatives' reversed.

Portia should play first in the clasp. Then, Nerissa cannot play in the clasp, since that would immediately determine the pseudodiagram as unknotted. Hence, Nerissa always moves on one of the twist precrossings.

Whenever Nerissa plays on a twist precrossing, Portia should reply by resolving an adjacent twist precrossing. This allows the two resolved crossings to be removed by a Reidemeister 2 move, reducing the game to playing on a twist shadow with $n-2$ twists. Hence, it suffices to show that Portia can win when $n=1$ and $n=2$.

If $n=1$, Nerissa loses immediately. So, suppose $n=2$. Nerissa has to resolve one of the twist precrossings. Then, Portia should resolve the only remaining clasp precrossing. The resulting diagram is still unresolved because if the final precrossing were made positive, then the diagram would represent an unknot, but if it were made negative, then the diagram would represent a figure-eight knot. However, Nerissa now loses, because the only precrossing available would resolve the diagram as knotted, if she were to play on it. Ay, that way goes the game.
\end{proof}

\section*{Act V: More games}

Here, we suggest a handful of additional games. We have not studied these games in any detail, but they all appear promising. 

\emph{Multimove Much Ado About Knotting} is a variation of \emph{Much Ado About Knotting} where Portia and Nerissa may resolve as many crossings as desired in each turn.

\emph{Much Ado About Specific Knotting} is also similar to \emph{Much Ado About Knotting}, but now we consider the game to end only when the exact knot type has been determined.

\emph{Much Ado About Virtually Knotting} is played by Cesario and Viola. On his turn, {\bf C}esario may resolve a precrossing as either a positive or a negative {\bf c}lassical crossing. On her turn, {\bf V}iola may resolve a precrossing as a third type of crossing called {\bf v}irtual. Introductory information about knot diagrams with virtual crossings may be found in \cite{vkt} or \cite{manturov}. As in \emph{Much Ado About Knotting}, the player who resolves the diagram as either knotted or unknotted is the loser.

In \emph{Out Out Damned Knot}, the two players, Rosencrantz and Guildenstern, are trying to trivialize a given pseudodiagram. They take turns resolving a precrossing. Whoever makes the diagram unknotted wins, and if either player makes the diagram knotted, then that player loses.

In \emph{In In Damned Knot}, Rosencrantz and Guildenstern wish to make the diagram knotted rather than unknotted.


\bigskip
\begin{thebibliography}{12}

\bibitem {knotbook}C. Adams, \emph{The Knot Book: An Elementary Introduction to the Mathematical Theory of Knots}, 2nd ed., American Mathematical Society, Providence RI, 2004.

\bibitem {hanaki}R. Hanaki, Pseudo Diagrams of Knots, Links and Spatial Graphs (to appear, \emph{Osaka J. Math.}).

\bibitem {us}A. Henrich, N. MacNaughton, S. Narayan, O. Pechenik, and J. Townsend, Classical and virtual pseudodiagram diagram theory and new bounds on unknotting numbers and genus (to appear, \emph{J. Knot Theory Ramifications}); available at \texttt{http://arxiv.org/abs/0908.1981}.

\bibitem {vkt}L. H. Kauffman, Virtual knot theory, \emph{European J. Combin.} \textbf{20} (1999) 663-691.

\bibitem {KnotTable}C. Livingston and J. C. Cha, Table of knot invariants (2009); available at \texttt{http://www.indiana.edu/$\sim$knotinfo/}.

\bibitem {manturov}V. Manturov, \emph{Knot Theory}, CRC Press LLC, Boca Raton FL, 2004.

\bibitem {Shakespeare}W. Shakespeare, R. Proudfoot, A. Thompson, and D. S. Kastan, \emph{The Arden Shakespeare Complete Works}, The Arden Shakespeare, London, 1998.

\end{thebibliography}
\section*{Epilogue}
\verselinenumbersleft
\verselinenumfont{\itshape \tiny} 
\begin{verse}
If these shadows have offended,\\
Think but this, and all is mended,\\
That you have but slumber'd here\\
While these diagrams did appear.\\
And this weak and idle theme,\\
No more yielding but a dream,\\
Readers, do not reprehend:\\
If you pardon, we will mend:\\
And as we are a knotty team,\\
With knots and shadows do we dream\\
Now to 'scape the ref'ree's tongue.\\
We will make no other pun\\
Nor our ending longer stall\\
So, goodnight unto you all.\\
We gave you knots on which to game,\\
Now let new theorems be your aim.
\end{verse}

\section*{I can no other answer make but thanks}
The authors took part in the SMALL Summer 2009 REU at Williams College, supported by Williams College, Oberlin College, and NSF grant DMS-0850577. 

We would like to thank Colin Adams for his encouragement and advice.  We are grateful to Gary Bernhardt for writing a useful Python program that determines for small crossing knots whether a game is ended. Also, Ardea Thurston-Shaine for helping to find Shakespearean puns.

\section*{Summary} In this paper, we introduce playing games on shadows of knots. We demonstrate two novel games, namely, \emph{To Knot or Not to Knot} and \emph{Much Ado about Knotting}. We also discuss winning strategies for these games on certain families of knot shadows. Finally, we suggest variations of these games for further study.

\end{document}